\documentclass[12pt]{article}
\usepackage{amssymb,amsmath,amsopn,amsfonts}
\usepackage{latexsym}
\usepackage{array}
\usepackage{graphicx}
\usepackage{amsthm}
\usepackage{colordvi,multicol}
\usepackage{epsf}
\usepackage{tikz}
\usepackage[T1]{fontenc}
\usepackage[utf8]{inputenc}
\usepackage{authblk}
\usepackage{hyperref}
\makeatletter
\newcommand\RSloop{\@ifnextchar\bgroup\RSloopa\RSloopb}
\makeatother
\newcommand\RSloopa[1]{\bgroup\RSloop#1\relax\egroup\RSloop}
\newcommand\RSloopb[1]%
  {\ifx\relax#1%
   \else
     \ifcsname RS:#1\endcsname
       \csname RS:#1\endcsname
     \else
       \GenericError{(RS)}{RS Error: operator #1 undefined}{}{}%
     \fi
   \expandafter\RSloop
   \fi
  }
\newcommand\X{0}
\newcommand\RS[1]%
  {\begin{tikzpicture}
     [every node/.style=
       {circle,draw,fill,minimum size=1.5pt,inner sep=0pt,outer sep=0pt},
      line cap=round
     ]
   \coordinate(\X) at (0,0);
   \RSloop{#1}\relax
   \end{tikzpicture}
  }
\newcommand\RSdef[1]{\expandafter\def\csname RS:#1\endcsname}
\newlength\RSu
\RSdef{i}{\draw (\X) -- +(90:\RSu) node{};}
\RSdef{l}{\draw (\X) -- +(135:\RSu) node{};}
\RSdef{r}{\draw (\X) -- +(45:\RSu) node{};}
\RSdef{I}{\draw (\X) -- +(90:\RSu) coordinate(\X I);\edef\X{\X I}}
\RSdef{L}{\draw (\X) -- +(135:\RSu) coordinate(\X L);\edef\X{\X L}}
\RSdef{R}{\draw (\X) -- +(45:\RSu) coordinate(\X R);\edef\X{\X R}}

\newcommand{\normmm}[1]{{\left\vert\kern-0.25ex\left\vert\kern-0.25ex\left\vert #1 
    \right\vert\kern-0.25ex\right\vert\kern-0.25ex\right\vert}}

\newtheorem{lemma}{Lemma}
\newtheorem{Thm}{Theorem}
\newtheorem*{remark}{Remark}
\newtheorem*{definition}{Definition}

\newfam\msbfam
\font\tenmsb=msbm10 \textfont\msbfam=\tenmsb \font\sevenmsb=msbm7
\scriptfont\msbfam=\sevenmsb \font\fivemsb=msbm5
\scriptscriptfont\msbfam=\fivemsb

\numberwithin{equation}{section}
\numberwithin{lemma}{section}
\numberwithin{Thm}{section}

\providecommand{\keywords}[1]{\textbf{\textbf{Keywords:}} #1}

\begin{document}
\bigbreak

\title{\bf On the small time asymptotics of the dynamical $\Phi^4_1$ model \thanks{Research supported in part  by NSFC (No.11771037). Financial supported by the DFG through the CRC 1283 "Taming uncertainty and profiting from randomness and low regularity in analysis, stochastics and their applications" is acknowledged.}}
\author[1,2]{\bf Bingguang Chen\thanks{bchen@math.uni-bielefeld.de}}
\author[1,2]{\bf Xiangchan Zhu\thanks{Corresponding author: zhuxiangchan@126.com}}
\date{}
\affil[1]{Academy of Mathematics and System Science, Chinese Academy of Science, Beijing 100190, China }
\affil[2]{Department of Mathematics, University of Bielefeld, D-33615 Bielefeld, Germany}
\renewcommand*{\Affilfont}{\small\it}
\renewcommand\Authand{, }

\maketitle

\begin{abstract}

In this paper,  we establish a small time large deviation principle (small time asymptotics) for the   
dynamical $\Phi^4_1$ model, which not only involves study of the space-time white noise with intensity $\sqrt{\varepsilon}$, but also the investigation of the effect of the small (with $\varepsilon$)  nonlinear drift. 

\end{abstract}


\vspace{2mm}


\section{Introduction}

In this paper we study small time behaviour of the dynamical $\Phi^4_1$ model :
\begin{equation}\label{1.1}\aligned
&du(t)=\Delta u(t)dt - u(t)^3dt+ dW(t), \text{ } \text{for}\text{ } (t,x)\in[0,T]\times \mathbb{T},\\
&u(0)=u_0,
\endaligned
\end{equation}
where $\mathbb{T}$ is one dimensional torus and $W $ is a cylindrical Wiener process on  $L^2(\mathbb{T})$.  

Equation (\ref{1.1}) in $d$ dimensional case describes the natural reversible dynamics for the Euclidean $\Phi^4_d$ quantum field theory. It is formally given by the following probability measure 
$$\nu(d\phi)=N^{-1}\prod_{x\in \mathbb{T}^d}d\phi(x)\exp[-\int_{\mathbb{T}^d}(\frac{1}{2}|\nabla\phi(x)|^2+\phi^4(x))dx],$$
where $N$ is a renormalization constant and $\phi$ is the real-valued field. This measure was  investigated intensively in the 1970s and 1980s (see \cite{Glimm87} and the references therein).  Parisi and Wu in  \cite{PW81} proposed a program named stochastic quantization of getting the measure as limiting distributions of stochastic processes, especially as solutions to nonlinear stochastic differential equations(see \cite{JLM85}). The issue to study $\Phi^4_d$ measure is to solve and study properties of (\ref{1.1}) in $d$ dimensional case. 

The dynamical $\Phi^4_1$ model with Dirichlet boundary condition (also named as reaction-diffusion equations) was studied systematically in \cite{DP04}. In \cite{DP04} not only existence and uniqueness of solutions to this equation have been obtained, but also the strong Feller property and ergodicity. For more details and more properties we refer to \cite[Section 4]{DP04}. We can obtain the results on the torus case similarly.

In 2 and 3 dimensions, the equation (\ref{1.1}) falls in the category of the singular SPDEs  due to the irregular nature of the noise $dW(t)$. Solutions are expected to take value in distribution spaces of negative regularity, which means the cubic term in the equation is not well-defined in the classical sense and renormalization has to be done for the nonlinear term.

In two spatial dimensions, weak solutions to (\ref{1.1}) have been first  constructed in \cite{AR91} by using Dirichlet form theory. In \cite{DD03} the authors decomposed (\ref{1.1}) into the linear equation and a shifted equation (so called Da Prato-Debussche trick) and obtain a probalistical strong solution via a fixed-point argument and invariant measure $\nu(d\phi)$. Recently, global well-posedness to (\ref{1.1}) via a PDE argument has been obtained in \cite{MW15}. See also \cite{RZZ15} for a study of relation between weak solutions and strong solutions.

By Hairer’s  breakthrough work on regularity structures \cite{Hai14}, (\ref{1.1}) in the three dimensional case is well-defined and local existance and uniqueness can be obtained.  In \cite{GIP15}  Gubinelli, Imkeller and Perkowski  introduced paracontrolled distributions method for singular SPDEs and by this method in  \cite{CC18} the authors  also obtained local-in-time well-posedness result. Mourrat and Weber in \cite{MW17CDFI} gave existence and uniqueness of global-in-time solutions on $\mathbb{T}^3$ by energy estimate and mild formulation. Recently, Gubinelli and Hofmanová in \cite{GH18} proved the global existence and uniqueness results for (\ref{1.1}) on $\mathbb{R}^3$ based on maximum principle and localization technique.

The purpose of this paper is to study the small time asymptotics (large deviations) of the dynamical $\Phi^4_1$ model.  We try to estimate the limiting behavior of the solution in time interval $[0,t]$ as $t$ goes to zero,  which describes how fast the solution approximating its initial data in sense of probability. The small time asymptotics in this case is also theoretically interesting, since the study involves the investigation of the small rough noise and the effect of the small nonlinear drift. The study of the small time asymptotics of finite dimensional diffusion processes was initiated by Varadhan in the influential work \cite{Var67}. The small time asymptotics (large deviation) of SPDEs were studied in \cite{Zh00}, \cite{XZ08}, \cite{LRZ13} and references therein.

We also want to mention the following small time asymptotics result by Dirichlet form. By \cite{AR91} and \cite{ZZ17} we know that the dynamical $\Phi^4_d$ model associated with a conservative and local Dirichlet form. Then the main result in \cite{HR03} implies the following Varadhan-type small time asymptotics for the dynamical $\Phi^4_d$ model: 
$$\lim_{t\rightarrow 0}t \log P^\nu(u(0)\in A, u(t)\in B)=-\frac{d(A,B)^2}{2},$$
for all measurable sets $A, B$, where $d$ is the intrinsic metric associated with the Dirichlet form of $\Phi^4_d$ model (see \cite{HR03} for the definition). However, these results is for the stationary case or holds for $\nu(d\phi)$-almost every starting point (see \cite[Theorem 1.3]{HR03} for a stronger version). The small time large deviation results in our paper hold for every starting point and is of its own interest.

Let $\varepsilon>0$, by the scaling property of the Brownian motion, it is easy to see that $u(\varepsilon t)$ coincides in law with the solution of the following equation:
\begin{equation}\label{1.2}\aligned
&du_\varepsilon=\varepsilon\Delta u_\varepsilon dt-\varepsilon u_\varepsilon^3dt+\sqrt{\varepsilon}dW,\\
&u_\varepsilon(0)=u_0.
\endaligned
\end{equation}
To establish the small time large deviation, we follow the idea of \cite{XZ08} to prove the solution to (\ref{1.2}) is exponentially equivalent to the solution to the linear equation. In our case, due to  the irregularity of the  white noise,  the It\^o formula in \cite{XZ08} cannot be uesd. Our calculations are based on the energy estimate for the shifted equation (see (\ref{shifted equation})) and the mild formulation.  

In \cite{HW15} the small noise large deviation principle for the dynamical $\Phi^4_d$ model is established. The authors considered the solution  as a continuous map $F$ of the noise $\sqrt{\varepsilon}\xi$ and some renormalization terms which belong to the Wiener chaos with the help  of the regularity structure, then the  result follows from the large deviation for Wiener chaos and the contraction principle. However, this method seems not work for the small time asymptotics problem. By this method, we have to prove the large deviation principle for the solution to linear equations in a better space (compared to Theorem \ref{weak result} in our paper), which seems not true since $e^{\varepsilon\Delta}\rightarrow I$ as $\varepsilon\rightarrow 0$ and the smoothing effect of heat flow will disappear. We will also meet this problem for the higher dimensional case which we will try to solve in the future.

 \vskip.10in

{\textbf{Organization of the paper}}

In Section 2, we introduce the basic  notation and recall some preliminary results.  In Section 3, we give the definition of large deviation principle, and prove the small time asymptotics for the linear equation. The small time large deviation for the dynamical $\Phi^4_1$ model (Theorem \ref{1d case}) is established in Section 4. 

 \vskip.10in

{\textbf{Acknowledgement}}

The authors would like to thank Rongchan Zhu for helpful discussions and also Peter Friz for pointing out \cite{HR03} to us.

\section{Preliminary}

In the following we recall some definitions of Besov spaces. For a general introduction to the theory we refer to \cite{BCD11}, \cite{Tri78}, \cite{Tri06}. First we introduce the following
notations. 
 The space of real valued infinitely differentiable functions of compact support is denoted by $\mathcal{D}(\mathbb{R}^d)$ or $\mathcal{D}$. The space of Schwartz functions is denoted by $\mathcal{S}(\mathbb{R}^d)$. Its dual, the space of tempered distributions, is denoted by $\mathcal{S}'(\mathbb{R}^d)$. The Fourier transform and the inverse Fourier transform are denoted by $\mathcal{F}$ and $\mathcal{F}^{-1}$, respectively.

 Let $\chi,\theta\in \mathcal{D}$ be nonnegative radial functions on $\mathbb{R}^d$, such that

i. the support of $\chi$ is contained in a ball and the support of $\theta$ is contained in an annulus;

ii. $\chi(z)+\sum_{j\geqslant0}\theta(2^{-j}z)=1$ for all $z\in \mathbb{R}^d$;

iii. $\textrm{supp}(\chi)\cap \textrm{supp}(\theta(2^{-j}\cdot))=\emptyset$ for $j\geqslant1$ and $\textrm{supp}\theta(2^{-i}\cdot)\cap \textrm{supp}\theta(2^{-j}\cdot)=\emptyset$ for $|i-j|>1$.

We call such $(\chi,\theta)$ dyadic partition of unity, and for the existence of dyadic partitions of unity we refer to \cite[Proposition 2.10]{BCD11}. The Littlewood-Paley blocks are now defined as
$$\Delta_{-1}u=\mathcal{F}^{-1}(\chi\mathcal{F}u)\quad \Delta_{j}u=\mathcal{F}^{-1}(\theta(2^{-j}\cdot)\mathcal{F}u).$$

 \vskip.10in

{\textbf{Besov spaces}}

For $\alpha\in\mathbb{R}$, $p,q\in [1,\infty]$, $u\in\mathcal{D}$ we define
$$\|u\|_{B^\alpha_{p,q}}:=(\sum_{j\geqslant-1}(2^{j\alpha}\|\Delta_ju\|_{L^p})^q)^{1/q},$$
with the usual interpretation as $l^\infty$ norm in case $q=\infty$. The Besov space $B^\alpha_{p,q}$ consists of the completion of $\mathcal{D}$ with respect to this norm and the H\"{o}lder-Besov space $\mathcal{C}^\alpha$ is given by $\mathcal{C}^\alpha(\mathbb{R}^d)=B^\alpha_{\infty,\infty}(\mathbb{R}^d)$. For $p,q\in [1,\infty)$,
$$B^\alpha_{p,q}(\mathbb{R}^d)=\{u\in\mathcal{S}'(\mathbb{R}^d):\|u\|_{B^\alpha_{p,q}}<\infty\}.$$
$$\mathcal{C}^\alpha(\mathbb{R}^d)\varsubsetneq \{u\in\mathcal{S}'(\mathbb{R}^d):\|u\|_{\mathcal{C}^\alpha(\mathbb{R}^d)}<\infty\}.$$
We point out that everything above and everything that follows can be applied to distributions on the torus (see \cite{S85}, \cite{SW71}). More precisely, let $\mathcal{S}'(\mathbb{T}^d)$ be the space of distributions on $\mathbb{T}^d$.  Besov spaces on the torus with general indices $p,q\in[1,\infty]$ are defined as
the completion of $C^\infty(\mathbb{T}^d)$ with respect to the norm $$\|u\|_{B^\alpha_{p,q}(\mathbb{T}^d)}:=(\sum_{j\geqslant-1}(2^{j\alpha}\|\Delta_ju\|_{L^p(\mathbb{T}^d)})^q)^{1/q},$$
and the H\"{o}lder-Besov space $\mathcal{C}^\alpha$ is given by $\mathcal{C}^\alpha=B^\alpha_{\infty,\infty}(\mathbb{T}^d)$.  We write $\|\cdot\|_{\alpha}$ instead of $\|\cdot\|_{B^\alpha_{\infty,\infty}(\mathbb{T}^d)}$ in the following for simplicity.  For $p,q\in[1,\infty)$
$$B^\alpha_{p,q}(\mathbb{T}^d)=\{u\in\mathcal{S}'(\mathbb{T}^d):\|u\|_{B^\alpha_{p,q}(\mathbb{T}^d)}<\infty\}.$$
$$ \mathcal{C}^\alpha\varsubsetneq \{u\in\mathcal{S}'(\mathbb{T}^d):\|u\|_{\alpha}<\infty\}.$$

Here we choose Besov spaces as  completions of smooth functions, which ensures that the Besov spaces are separable which has a lot of advantages for our analysis below.

In this paper, we use the following notations: $C\mathcal{C}^\beta:=C([0,T],\mathcal{C}^\beta)$, $CL^\infty:=C([0,T],L^\infty(\mathbb{T}^d))$.

 \vskip.10in
\textbf{Estimates on the torus}

In this part we give estimates on the torus for later use.
We  will need several important properties of Besov spaces on the torus and we recall the following Besov embedding theorems on the torus first (c.f. \cite[Theorem 4.6.1]{Tri78}, \cite[ Lemma A.2]{GIP15}):
\vskip.10in
\begin{lemma} \label{lemma1}
 Let $1\leq p_1\leq p_2\leq\infty$ and $1\leq q_1\leq q_2\leq\infty$, and let $\alpha\in\mathbb{R}$. Then $B^\alpha_{p_1,q_1}(\mathbb{T}^d)$ is continuously embedded in $B^{\alpha-d(1/p_1-1/p_2)}_{p_2,q_2}(\mathbb{T}^d)$.

\end{lemma}
We recall the following Schauder estimates, i.e. the smoothing effect of the heat flow, for later use.

\vskip.10in
\begin{lemma}[{\cite[Lemma A.7]{GIP15}}]\label{lemma2}
 Let $u\in \mathcal{C}^\alpha$ for some $\alpha\in \mathbb{R}$. Then for every $\delta\geqslant 0$, there exists a constant $C$ independent of $u$ such that
$$\|e^{t\Delta}u\|_{\alpha+\delta}\leqslant Ct^{-\delta/2}\|u\|_{\alpha}.$$
\end{lemma}

\section{Large deviation principle and some preparations}

\subsection{Large deviation principle}
We recall the definition of the large deviation principle. For a general introduction to the theory we refer to \cite{DZ92}, \cite{DZ98}.

\begin{definition}
Given a family of probability measures $\{\mu_\varepsilon\}_{\varepsilon>0}$ on a complete separable metric space $(E,\rho)$ and a lower semicontinuous function $I:E\rightarrow [0,\infty]$, not identically equal to $+\infty$. The family $\{\mu_\varepsilon\}$ is said to satisfy the large deviation principle(LDP) with respect to the rate function $I$ if\\
(U) for all closed sets $F\subset E$ we have 
$$\limsup_{\varepsilon\rightarrow 0}\varepsilon\log\mu_\varepsilon(F)\leqslant-\inf_{x\in F}I(x),$$
(L) for all open sets $G\subset E$ we have 
$$\liminf_{\varepsilon\rightarrow 0}\varepsilon\log\mu_\varepsilon(G)\geqslant-\inf_{x\in G}I(x).$$

A family of random variable is said to satisfy large deviation principle if the law of these random variables satisfy large deviation princple.

Moreover, $I$ is a good rate function if its level sets $I_r:=\{x\in E:I(x)\leqslant r\}$ are compact for arbitrary $r\in (0,+\infty)$.
\end{definition}
 \vskip.10in

Given a probabilty space $(\Omega,\mathcal{F},P)$, the random variables $\{Z_\varepsilon\}$ and $\{\overline{Z}_\varepsilon\}$ which take values in $(E,\rho)$ are called exponentially equivalent if for each $\delta>0$, 
$$\lim_{\varepsilon\rightarrow 0}\varepsilon\log P(\rho(Z_\varepsilon,\overline{Z}_\varepsilon)>\delta)=-\infty.$$
\vskip.10in
\begin{lemma}[{\cite[Theorem 4.2.13]{DZ98}}]\label{EXEQ}
If an LDP with a rate function $I(\cdot)$ holds for the random variables $\{Z_\varepsilon\}$, which are exponentially equivalent to $\{\overline{Z}_\varepsilon\}$, then the same LDP holds for $\{\overline{Z}_\varepsilon\}$.
\end{lemma}

\subsection{Small time asymptotics in the linear case}

In this subsection we concentrate on  the following linear equations on the torus $\mathbb{T}$: 

\begin{equation}\label{linear}\aligned
&dZ_\varepsilon(t)=\varepsilon \Delta Z_\varepsilon(t)dt+\sqrt{\varepsilon}dW(t),\\
&Z_\varepsilon(0)=z_0.\\
\endaligned
\end{equation}
where $W(t)$ is an $L^2(\mathbb{T})$ cylindrical Wiener process and $z_0\in \mathcal{C}^{-\beta}$ for $0<\beta<\frac{1}{4}$.  We will prove that the solutions to ($\ref{linear}$) satisfy a large deviation principle.
\vskip.10in

The mild solutions to $(\ref{linear})$ are given by

$$Z_\varepsilon(t)=e^{\varepsilon t\Delta}z_0+\sqrt{\varepsilon}\int^t_0e^{\varepsilon(t-s)\Delta}dW(s).$$

\begin{Thm}\label{weak result}
Assume $z_0\in\mathcal{C}^{-\beta}$ for $0<\beta<\frac{1}{4}$. Let $\mu_{\varepsilon, z_0}=\mathcal{L}(Z_\varepsilon(\cdot))$ and $\alpha>0$ small enough. Define a functional $I$ on $C\mathcal{C}^{-\frac{1}{2}-\alpha}$ by  
$$I^{z_0}(g)=\inf_{h\in \Gamma_g}\{\frac{1}{2}\int^T_0\|h'(t)\|^2_{L^2(\mathbb{T})}dt\},$$
where 
$$\Gamma_g=\{h\in C\mathcal{C}^{-\frac{1}{2}-\alpha}: h(\cdot)\text{ } \text{is}\text{ } \text{absolutely} \text{ }\text{continuous},\text{ } g(t)=z_0+\int^t_0h'(s)ds\}.$$

 Then $\mu_{\varepsilon,z_0}$ satisfies a large deviation principle with the rate function $I^{z_0}(\cdot)$.
 
 Moreover, $I^{z_0}$ is a good rate function.
\end{Thm}

\begin{proof}

Let $x_\varepsilon$ be the solution to the stochastic equation
$$x_\varepsilon(t)=z_0+\sqrt{\varepsilon}\int^t_0dW(s).$$

Since $x_\varepsilon$ is Gaussian on $C\mathcal{C}^{-\frac{1}{2}-\alpha}$, by \cite[Theorem 12.9]{DZ92}, we know that $x_{\varepsilon}-z_0$ satisfy a large deviation principle with the rate function $I^{0}$.  Combing the deterministic initial data, we deduce that $x_\varepsilon$ satisfy a large deviation principle with the rate function $I^{z_0}$.

Now we  prove that $I^{z_0}$ is a good rate function. Consider the level set for $r\in (0,\infty)$ $$I^{z_0}_r=\{g\in C\mathcal{C}^{-\frac{1}{2}-\alpha}: I^{z_0}(g)\leqslant r\}.$$ 
For any $g\in I^{z_0}_r$, we have for $s,t\in[0, T]$
$$\|g(t)-g(s)\|_{-\frac{1}{2}-\alpha}\leqslant C\|g(t)-g(s)\|_{L^2(\mathbb{T})}\leqslant C\int^t_s\|g'(l)\|_{L^2(\mathbb{T})}dl\leqslant C(2r)^\frac{1}{2}|t-s|^\frac{1}{2},$$
where we use Lemma \ref{lemma1} in the first inequality and H\"older's inequality in the last inequality. Since the constant $C$ does not depend on $g$,  $I_r^{z_0}$ is equicontinuous. 
For each $t\in[0,T]$, let $I^{z_0}_{r,t}:=\{g(t),g\in   I^{z_0}_r\}$. For any $a\in I^{z_0}_{r,t}$, there exists $g\in I^{z_0}_r$ such that $a=g(t)$. Then H\"older's inequality implies
$$\|a-z_0\|_{L^2(\mathbb{T})}=\|g(t)-g(0)\|_{L^2(\mathbb{T})}\leqslant Cr^\frac{1}{2}.$$
Thus $I^{z_0}_{r,t}$ is contained in a ball $B_{L^2}(z_0, Cr^\frac{1}{2})$. By \cite[Proposition 4.6]{Tri06}, the embedding $L^2(\mathbb{T}) \hookrightarrow \mathcal{C}^{-\frac{1}{2}-\alpha}$ is compact, which implies that $I^{z_0}_{r,t}$ is relatively compact in $\mathcal{C}^{-\frac{1}{2}-\alpha}$ for any $t$.  Then the generalized Aerel\`a-Ascoli theorem implies that $I_r^{z_0}$ is compact, i.e., $I^{z_0}$ is a good rate function.

By Lemma $\ref{EXEQ}$, the task remain is to show that $Z_{\varepsilon}$ and $x_{\varepsilon}$ are exponentially equivalent, that is, for any $\delta>0$, 
$$\lim_{\varepsilon\rightarrow 0}\varepsilon\log P(\sup_{0\leqslant t\leqslant T}\|Z_\varepsilon(t)-x_\varepsilon(t)\|_{{-\frac{1}{2}-\alpha}}>\delta)=-\infty.$$

Let $w_\varepsilon=Z_\varepsilon-x_\varepsilon$, we have 
$$\frac{d}{dt}w_\varepsilon(t)=\varepsilon\Delta w_\varepsilon(t)+\varepsilon\Delta x_\varepsilon(t), \text{ } \text{ }w_\varepsilon(0)=0.$$

The mild formulation of $w_\varepsilon$ is given by 
$$\aligned
w_\varepsilon(t)&=\varepsilon\int^t_0e^{\varepsilon(t-s)\Delta}\Delta x_\varepsilon(s)ds\\
    &=\varepsilon\int^t_0e^{\varepsilon(t-s)\Delta}\Delta z_0ds +\varepsilon\sqrt{\varepsilon}\int^t_0e^{\varepsilon(t-s)\Delta}\Delta W(s)ds.\\
\endaligned$$

Now we estimate every term in the second line. By Lemma \ref{lemma2}, we have 
$$\aligned
\sup_{0\leqslant t\leqslant T}\|\varepsilon\int^t_0e^{\varepsilon(t-s)\Delta}\Delta z_0ds\|_{-\frac{1}{2}-\alpha}&\leqslant \sup_{0\leqslant t\leqslant T}C\varepsilon\int^t_0 \frac{1}{[\varepsilon(t-s)]^{\frac{3}{4}-\frac{\alpha-\beta}{2}}}\|\Delta z_0\|_{-2-\beta}ds\\
&\leqslant C\varepsilon^{\frac{1}{4}+\frac{\alpha-\beta}{2}}\|z_0\|_{-\beta}.
\endaligned$$

Similarly, we have for $0<\kappa_1<\frac{\alpha}{2}$, 
\begin{align*}
\sup_{0\leqslant t\leqslant T}\|\varepsilon\sqrt{\varepsilon}\int^t_0e^{\varepsilon(t-s)\Delta}\Delta W(s)ds\|_{-\frac{1}{2}-\alpha}&\leqslant  \sup_{0\leqslant t\leqslant T}C\varepsilon\sqrt{\varepsilon}\int^t_0 \frac{1}{[\varepsilon(t-s)]^{1-\kappa_1}}\|\Delta W(s)\|_{-\frac{5}{2}-\alpha+2\kappa_1}ds\\
  &\leqslant C \sqrt{\varepsilon}\varepsilon^{\kappa_1}\sup_{0\leqslant t\leqslant T}\|W(t)\|_{-\frac{1}{2}-\alpha+2\kappa_1}.
\end{align*}

We should point out that the constant $C$ above is independent of $\varepsilon$ and may change from line to line.

For the cylindrical Wiener process $W$,  we have for $s,t \in[0, T]$, $0<\kappa_1<\frac{\alpha}{3}$

\begin{align*}
E|\triangle_j(W(t)-W(s))|^2=E|\sum_{k\in\mathbb{Z}}\theta_j(k)e_k\langle W(t)-W(s), e_k\rangle|^2\\
\leqslant C|t-s|(1+\sum_{k\in\mathbb{Z}\setminus\{0\}}\frac{2^{j(1+2\alpha-6\kappa_1)}}{|k|^{1+2\alpha-6\kappa_1}})\leqslant C|t-s|2^{j(1+2\alpha-6\kappa_1)},
\end{align*}
where $e_k=2^{-\frac{1}{2}}e^{i\pi kx}$ and we use $k\in\text{supp}\theta_j\subset 2^j\mathcal{A}$($\mathcal{A}$ is an annulus).

By Nelson’s hypercontractive estimate in  \cite{Nel73}, for $p>2$, there exists a constant $C$ independent of $p$ such that
$$\aligned
E\|\triangle_j (W(t)-W(s))\|^p_{L^p(\mathbb{T})}&=\int E|\triangle_j (W(t)-W(s))|^p(x)dx\\
&\leqslant C^pp^{\frac{p}{2}}\int (E|\triangle_j (W(t)-W(s))|^2(x))^{\frac{p}{2}}dx.
\endaligned$$

Then we obtain for $\frac{1}{p}<\kappa_1$
$$E\|W(t)-W(s)\|^p_{B_{p,p}^{-\frac{1}{2}-\alpha+2\kappa_1+\frac{1}{p}}(\mathbb{T})}\leqslant C^p|t-s|^{\frac{p}{2}}p^{\frac{p}{2}}\sum_{j\geqslant -1}2^{j(-\kappa_1+\frac{1}{p}) p}.$$

Thus Lemma \ref{lemma1} and Kolmogorov's continuity criterion imply that for $p>\frac{1}{\kappa_1}$
$$(E[\sup_{0\leqslant t\leqslant T}\|W\|^p_{-\frac{1}{2}-\alpha+2\kappa_1}])^{\frac{1}{p}}\leqslant C(E[\sup_{0\leqslant t\leqslant T}\|W\|^p_{B^{-\frac{1}{2}-\alpha+2\kappa_1+\frac{1}{p}}_{p,p}(\mathbb{T})}])^{\frac{1}{p}}\leqslant Cp^\frac{1}{2}.$$

Hence, with the above estimates in hand, we have

\begin{align*}
(E\sup_{0\leqslant t\leqslant T}\|w_\varepsilon(t)\|^{p}_{-\frac{1}{2}-\alpha})^\frac{1}{p}&\leqslant C\varepsilon^{\frac{1}{4}+\frac{\alpha-\beta}{2}}\|z_0\|_{-\beta}+C\sqrt{\varepsilon}\varepsilon^{\kappa_1}(E[\sup_{0\leqslant t\leqslant T}\|W\|_{-\frac{1}{2}-\alpha+2\kappa_1}]^p)^{\frac{1}{p}}\\
&\leqslant C\varepsilon^{\kappa_1}(1+\sqrt{\varepsilon}p^\frac{1}{2}),\\
\end{align*}
where $C$ is the constant independent of $\varepsilon, p$ and may change from line to line.

Therefore Chebyshev's inequality implies that 
\begin{align*}
\varepsilon\log P(\sup_{0\leqslant t\leqslant T}\|w_\varepsilon(t)\|_{{-\frac{1}{2}-\alpha}}>\delta)&\leqslant \varepsilon\log\frac{E\sup_{0\leqslant t\leqslant T}\|w_\varepsilon(t)\|^{p}_{-\frac{1}{2}-\alpha}}{\delta^p}\\
&\leqslant \varepsilon p(\log C\varepsilon^{\kappa_1}(1+ \sqrt{\varepsilon}p^{\frac{1}{2}})-\log\delta).\\
\end{align*}

Let $p=\frac{1}{\varepsilon}$ and $\varepsilon\rightarrow 0$, the proof is finished.
\end{proof}

 \vskip.10in

\section{Small time asymptotics for $\Phi^4_1$ model} 
In this section  we consider the equation

$$\aligned
  du(t) &= \Delta u(t)dt-u^3(t)dt+dW(t),\\
  u(0)&=u_0,\\
  \endaligned $$
where $u_0\in \mathcal{C}^{-\beta}$ for $0<\beta<\frac{1}{4}$ and $W$ is a cylindrical Wiener process on $L^2(\mathbb{T})$. By a similar argument as \cite[Theorem 4.8]{DP04}, we obtain that the equation  has a unique solution $u\in C\mathcal{C}^{-\beta}$. 

Let $\varepsilon>0$, by the scaling property of the Brownian motion, it is easy to see that $u(\varepsilon t)$ coincides in law with the solution to the following equation:

$$\aligned
  &du_\varepsilon =\varepsilon \Delta u_\varepsilon dt-\varepsilon u_\varepsilon^3dt+\sqrt{\varepsilon}dW,\\
  &u_{\varepsilon}(0)=u_0.\\
  \endaligned $$

Our purpose is to establish a large deviation principle for $u_\varepsilon$. The main result is the following Theorem:

\begin{Thm}\label{1d case}
Assume $u_0\in\mathcal{C}^{-\beta}$ for $0<\beta<\frac{1}{4}$ and $\alpha>0$ small enough, then $u_\varepsilon$ satisfies LDP on $C\mathcal{C}^{-\frac{1}{2}-\alpha}$ with the good rate function $I^{u_0}$, where $I^{u_0}$ is given in Theorem \ref{weak result}.
\end{Thm}

Let $Z_\varepsilon(t)$ be the solution to the linear equation with the same initial condition as $u$ :

$$\aligned
&dZ_\varepsilon(t)= \varepsilon \Delta Z_\varepsilon(t)dt+\sqrt{\varepsilon}dW(t), \\
&Z_\varepsilon(0)= u_0.\\
\endaligned$$

Theorem \ref{weak result} implies that $Z_\varepsilon$ satisfies a large deviation principle on the space $C\mathcal{C}^{-\frac{1}{2}-\alpha}$ with the rate function $I^{u_0}$. By Lemma \ref{EXEQ}, our task is to show that $u_\varepsilon$ and $Z_\varepsilon$ are exponentially equivalent in $C\mathcal{C}^{-\frac{1}{2}-\alpha}$. 

\subsection{Estimate of $Z_\varepsilon$}

In this subsection,  we follow the notations from \cite[ Section 9]{GP15} to estimate $Z_\varepsilon$:  We represent the white noise in terms of its spatial Fourier transform. Let $E=\mathbb{Z}\setminus \{0\}$ and let $W(s,k)= \langle W(s), e_k\rangle $, where $\{e_k:=2^{-\frac{1}{2}}e^{i\pi kx}\}_{k\in \mathbb{Z}}$ is the Fourier basis of $L^2(\mathbb{T})$. Here for simplicity we assume that $\langle W(s), e_0\rangle=0$ and restrict ourselves to the flow of $\int_{\mathbb{T}}u(x)dx=0$. In the following we view $W(s,k)$ as a Gaussian process on $\mathbb{R}\times E$ with covariance given by
$$E[\int_{\mathbb{R}\times E}f(\eta)w(d\eta)\int_{\mathbb{R}\times E}g(\eta')w(d\eta')]=\int_{\mathbb{R}\times E}f(\eta_1)g(\eta_{-1})d\eta,$$
where $\eta_a=(s_a,k_a)$ and the measure $d\eta_a=ds_adk_a$ is the product measure of the Lebesgue measure $ds$ on $\mathbb{R}$ and the counting measure $dk$ on $E$ .

 Let $\overline{Z}_\varepsilon=Z_\varepsilon-e^{\varepsilon{t}\Delta}u_0$, then 
$$\overline{Z}_\varepsilon(t,x)=\int_{\mathbb{R}\times E}\sqrt{\varepsilon}e_k(x)e^{-\varepsilon(t-s)\pi|k|^2}1_{\{0<s<t\}}W(d\eta).$$ 

Now we have the following calculations: for $s, t\in [0, T]$,
\begin{equation}\label{half estimate of Z}\aligned
&E[|\triangle_j (\overline{Z}_\varepsilon(t)-\overline{Z}_\varepsilon(s))|^2]\\
=&E[|\int \theta_j(k_1)( \sqrt{\varepsilon}e_{k_1}e^{-\varepsilon(t-s_1)\pi|k_1|^2}1_{\{0<s_1<t\}}-\sqrt{\varepsilon}e_{k_1}e^{-\varepsilon(s-s_1)\pi|k_1|^2}1_{\{0<s_1<s\}})W(d\eta_1)|^2]\\   
                              \leqslant &\varepsilon C\int\theta_j^2(k_1)(e^{-2\varepsilon(t-s_1)\pi|k_1|^2}1_{\{s<s_1<t\}}+|e^{-\varepsilon(t-s)\pi|k_1|^2}-1|^2e^{-2\varepsilon(s-s_1)\pi|k_1|^2}1_{\{0<s_1<s\}})d\eta_1\\                        
                              \leqslant &C\int\theta_j(k_1)^2\frac{(\varepsilon |t-s||k_1|^2)^{2\kappa}}{|k_1|^{2}}dk_1\\
                              \leqslant &C\varepsilon^{2\kappa}|t-s|^{2\kappa}2^{j}\frac{1}{(2^{j})^{2-4\kappa}}=C\varepsilon^{2\kappa}|t-s|^{2\kappa}2^{j(-1+4\kappa)},\\
\endaligned                             
\end{equation}
where we  use $1-e^x\leqslant|x|^\kappa$ for $\kappa\in(0,1), x<0$ in the fourth inequality and $k\in \text{supp} \theta_j\subset 2^j\mathcal{A}$($\mathcal{A}$ is an annulus) in the last inequality.  Here the constant $C$ is independent of $\varepsilon$ and may change from line to line.

By Nelson’s hypercontractive estimate in \cite{Nel73}, we have for $p>2$, there exists a constant $C$ independent of $p,\varepsilon$ such that

$$\aligned
E\|\triangle_j (\overline{Z}_\varepsilon(t)-\overline{Z}_\varepsilon(s))\|^p_{L^p(\mathbb{T})}&=\int E|\triangle_j (\overline{Z}_\varepsilon(t)-\overline{Z}_\varepsilon(s))|^p(x)dx\\
&\leqslant C^pp^{\frac{p}{2}}\int (E|\triangle_j (\overline{Z}_\varepsilon(t)-\overline{Z}_\varepsilon(s))|^2(x))^{\frac{p}{2}}dx.
\endaligned$$

Let $\kappa=\frac{1}{4}-\kappa'$ for $\kappa'>0$ small enough,  we obtain

$$
E\|\overline{Z}_\varepsilon(t)-\overline{Z}_\varepsilon(s)\|^p_{B^{\kappa'}_{p,p}(\mathbb{T})}\leqslant C^pp^{\frac{p}{2}}(\varepsilon|t-s|)^{(\frac{1}{4}-\kappa')p}.$$

Then Lemma $\ref{lemma1}$ and Kolmogorov's continuity criterion implies that for $p>\frac{1}{\kappa'}$, we have

\begin{equation}\label{estimate of Z}
E\|\overline{Z}_\varepsilon\|^p_{CL^\infty}\leqslant E\|\overline{Z}_\varepsilon\|^p_{C\mathcal{C}^{\kappa'-\frac{1}{p}}}\leqslant E\|\overline{Z}_\varepsilon\|^p_{C([0,T]; B^{\kappa'}_{p,p}(\mathbb{T}))}\leqslant C^p\varepsilon^{(\frac{1}{4}-\kappa')p}p^{\frac{p}{2}}.
\end{equation}

\begin{remark}
We want to emphasize that (\ref{estimate of Z}) only holds for $\overline{Z}_\varepsilon$ due to $\overline{Z}_\varepsilon(0)=0$. For the stationary one this does not hold since the expectation of  the stationary one does not depend on $\varepsilon$.
\end{remark}

\subsection{Exponentially equivalence}

To prove Theorem \ref{1d case}, by Lemma $\ref{EXEQ}$, we only need to prove the following theorem:

\begin{Thm}\label{ExEq}
For any $\delta>0$,

\begin{equation}
\lim_{\varepsilon\rightarrow0}\varepsilon\log P(\sup_{0\leqslant t\leqslant T}\|u_\varepsilon(t)-Z_\varepsilon(t)\|_{{-\frac{1}{2}-\alpha}}>\delta)=-\infty.
\end{equation}
\end{Thm}

\begin{proof}

At the beginning of the proof,  we should point out that the constant $C$ in the following is independent of $\varepsilon, p$ and may change from line to line.

Let $v_\varepsilon(t):=u_\varepsilon(t)-Z_\varepsilon(t)$, then $v_\varepsilon$ is the solution to the following shifted equation:

\begin{equation}\label{shifted equation}\aligned
&dv_\varepsilon(t)=\varepsilon \Delta v_\varepsilon(t)dt-\varepsilon (v_\varepsilon(t)+ Z_\varepsilon(t))^3dt,\\
&v_\varepsilon(0)=0.\\
\endaligned
\end{equation}

For $p\geqslant 1$, we have 
$$\aligned
\frac{1}{2p}\frac{d}{dt}\|v_\varepsilon\|^{2p}_{L^{2p}(\mathbb{T})}&=\varepsilon\langle \Delta v_\varepsilon, v_\varepsilon^{2p-1} \rangle-\varepsilon\langle v_\varepsilon^3 , v_\varepsilon^{2p-1}\rangle-3\varepsilon\langle v_\varepsilon^2 Z_\varepsilon , v_\varepsilon^{2p-1}\rangle-3\varepsilon\langle v_\varepsilon Z_\varepsilon^2 , v_\varepsilon^{2p-1}\rangle-\varepsilon\langle Z_\varepsilon^3 , v_\varepsilon^{2p-1}\rangle.\\
\endaligned$$

Then 
\begin{align*}
&\frac{1}{2p}\|v_\varepsilon(t)\|^{2p}_{L^{2p}(\mathbb{T})}+\varepsilon\int^t_0[(2p-1)\langle \nabla v_\varepsilon(s), v_\varepsilon^{2p-2}(s)\nabla v_\varepsilon(s)\rangle+\|v_\varepsilon^{2p+2}(s)\|_{L^1(\mathbb{T})}]ds\\
=& -\varepsilon\int^t_0 [3\langle v_\varepsilon^{2p+1}(s), Z_\varepsilon(s)\rangle+3\langle v_\varepsilon^{2p}(s), Z_\varepsilon^2(s)\rangle+\langle v_\varepsilon^{2p-1}(s), Z_\varepsilon^3(s)\rangle ]ds\\
\leqslant& \varepsilon\int^t_0 (a\|v_\varepsilon(s)^{2p+2}\|_{L^1(\mathbb{T})}+C\|Z_\varepsilon(s)\|^{2p+2}_{L^\infty(\mathbb{T})})ds,
\end{align*}
where we use H\"older's inequality and Young's inequality in the last inequality and $a\in(0,1)$. Take $p=3$, for $t\in[0, T]$, we have
\begin{equation}\label{estiamte to shifted equation}\aligned
\|v_\varepsilon(t)\|^6_{L^6(\mathbb{T})}&\leqslant\varepsilon C\int^t_0\|Z_{\varepsilon}(s)\|^8_{L^\infty(\mathbb{T})}ds\\
&\leqslant \varepsilon C\int^t_0(\|e^{\varepsilon s\Delta}u_0\|^8_{\beta'}+\|\overline{Z}_\varepsilon(s)\|^8_{L^\infty(\mathbb{T})})ds\\
&\leqslant \varepsilon C\int^t_0(\frac{1}{(\varepsilon s)^{\frac{8(\beta'+\beta)}{2}}}\|u_0\|^8_{-\beta}+\|\overline{Z}_\varepsilon(s)\|^8_{L^\infty(\mathbb{T})})ds\\
&\leqslant  C (\varepsilon^{1-4(\beta'+\beta)}\|u_0\|^8_{-\beta}+\varepsilon\|\overline{Z}_\varepsilon\|^8_{CL^\infty}),
\endaligned
\end{equation}
where $0<\beta'<\frac{1}{4}-\beta$  and we use Lemma \ref{lemma2} in the third inequality.

Thus  Young's inequality and  the mild formulation of $v_\varepsilon$ given by  
$$v_\varepsilon(t)=\varepsilon\int^t_0e^{\varepsilon(t-s)\Delta}[-v_\varepsilon^3-3v_\varepsilon^2Z_\varepsilon-3v_\varepsilon Z_\varepsilon^2-Z_\varepsilon^3]ds$$
imply that
$$\aligned
\sup_{t\in[0,T]}\|v_\varepsilon(t)\|_{L^2(\mathbb{T})}&\leqslant  \varepsilon C\int^T_0(\|v_\varepsilon(s)\|^3_{L^6(\mathbb{T})}+\|Z_\varepsilon(s)\|^3_{L^\infty(\mathbb{T})})ds\\
   &\leqslant  \varepsilon C\int^T_0(\|v_\varepsilon(s)\|^3_{L^6(\mathbb{T})}+\frac{1}{(\varepsilon s)^{\frac{3(\beta'+\beta)}{2}}}\|u_0\|^3_{-\beta}+\|\overline{Z}_\varepsilon(s)\|^3_{L^\infty(\mathbb{T})})ds\\
   &\leqslant  C(\varepsilon^{\frac{3}{2}-2(\beta'+\beta)}+\varepsilon^{\frac{3}{2}}\|\overline{Z}_\varepsilon\|^4_{C L^\infty}+\varepsilon^{1-\frac{3(\beta'+\beta)}{2}}+\varepsilon\|\overline{Z}_\varepsilon\|^3_{C L^\infty}),
\endaligned$$
where we use Lemma \ref{lemma2} in the second inequality and (\ref{estiamte to shifted equation}) in the last inequality.

 Thus by (\ref{estimate of Z}) we have for $3q>\frac{1}{\kappa'}$
$$\aligned
(E\sup_{t\in[0,T]}\|v_\varepsilon(t)\|^q_{L^2(\mathbb{T})})^\frac{1}{q}&\leqslant C(\varepsilon^{\frac{3}{2}-2(\beta'+\beta)}+\varepsilon^{\frac{3}{2}} (E[\|\overline{Z}_\varepsilon\|^{4q}_{C L^\infty}])^\frac{1}{q}+\varepsilon^{1-\frac{3(\beta'+\beta)}{2}}+\varepsilon (E[\|\overline{Z}_\varepsilon\|^{3q}_{C L^\infty}])^\frac{1}{q})\\
   &\leqslant  C(\varepsilon^{\frac{3}{2}-2(\beta'+\beta)}+\varepsilon^{\frac{5}{2}-4\kappa'}q^2+\varepsilon^{1-\frac{3(\beta'+\beta)}{2}}+\varepsilon^{\frac{7}{4}-3\kappa'}q^\frac{3}{2}).
\endaligned$$

Therefore, by Chebyshev's inequality  and Lemma \ref{lemma1} we have
\begin{align*}
&\varepsilon\log P(\sup_{0\leqslant t\leqslant T}\|v_\varepsilon(t)\|_{{-\frac{1}{2}-\alpha}}>\delta)\\
   \leqslant &\varepsilon\log\frac{E\sup_{t\in[0,T]}\|v_\varepsilon(t)\|^q_{L^2(\mathbb{T})}}{\delta^{q}}\\
   \leqslant &\varepsilon q[\log [C(\varepsilon^{\frac{3}{2}-2(\beta'+\beta)}+\varepsilon^{\frac{5}{2}-4\kappa'}q^2+\varepsilon^{1-\frac{3(\beta'+\beta)}{2}}+\varepsilon^{\frac{7}{4}-3\kappa'}q^\frac{3}{2})]-\log\delta].\\
\end{align*}

Let $q=\frac{1}{\varepsilon}$, we deduce that
$$\lim_{\varepsilon\rightarrow0}\varepsilon\log P(\sup_{0\leqslant t\leqslant T}\|v_\varepsilon(t)\|_{{-\frac{1}{2}-\alpha}}>\delta)=-\infty.$$
\end{proof}

Then Theorem $\ref{1d case}$ follows from Lemma \ref{EXEQ} and Theorem \ref{ExEq} .


\bibliography{Ref0}{}  
\bibliographystyle{alpha}


\end{document}